\documentclass[11pt]{amsart}
\usepackage{latexsym,amsmath,amssymb}
\newtheorem{thm}{Theorem}[section]

\newtheorem{lem}[thm]{Lemma}

\def\Z{{\mathbb Z}}
\def\Cay{{\rm Cay}}
\def\Circ{{\rm Circ}}
\def\Aut{{\rm Aut}}

\begin{document}
\title{Automorphisms of circulants that respect partitions}
\author[J.~Morris]{Joy Morris} 
\address{Department of Mathematics and Computer Science \\
University of Lethbridge \\
Lethbridge, AB. T1K 3M4. Canada}
\email{joy.morris@uleth.ca}

\thanks{This research was supported in part by the National Science
  and Engineering Research Council of Canada} 


\begin{abstract}
In this paper, we begin by partitioning the edge (or arc) set of a circulant (di)graph according to which generator in the connection set leads to each edge.  We then further refine the partition by subdividing any part that corresponds to an element of order less than $n$, according to which of the cycles generated by that element the edge is in.  It is known that if the (di)graph is connected and has no multiple edges, then any automorphism that respects the first partition and fixes the vertex corresponding to the group identity must be an automorphism of the group (this is in fact true in the more general context of Cayley graphs).  We show that automorphisms that respect the second partition and fix $0$ must also respect the first partition, so are again precisely the group automorphisms of $\Z_n$.
\end{abstract}
\maketitle

\section{Introduction}\label{intro}

In any Cayley digraph, there is a natural partition of the edge set according to the elements of the connection set that define them.  If $\Gamma=\Cay(G;S)$ where $S=\{s_1, \ldots, s_k\}$, then this natural partition is defined by $${\mathcal B}=\{\{(g,gs_i): g \in G\}: 1\le i \le k\}.$$

Now, any $s_i \in S$ generates a subgroup of $G$.  Let $G_{i,1},G_{i,2}, \ldots, G_{i,k_i}$ be the $k_i$ distinct cosets of this subgroup (and $G_{i,1}=\langle s_i \rangle$).  Then we can form a partition $\mathcal C$ that is a refinement of $\mathcal B$, with $${\mathcal C}=\{\{(g,gs_i): g \in G_{i,j}\}:1 \le j \le k_i, 1 \le i \le k\}.$$  Notice that each set in $\mathcal C$ consists of precisely the edges of a cycle all of whose edges are formed by a single element of $S$.

In the case of a Cayley graph, we replace each of the ordered pairs above with the corresponding unordered pair, and eliminate any duplication that may result (so $\mathcal B$ and $\mathcal C$ are sets, not multi-sets).

We say that an automorphism $\beta$ of a graph {\em respects} a partition $\{A_1, \ldots, A_n\}$ of the edge set of that graph, if $$\{A_1, \ldots, A_n\}=\{\beta(A_1), \ldots, \beta(A_n)\}.$$

It is little more than an observation to prove that in a connected Cayley digraph, any automorphism that respects the partition $\mathcal B$ and fixes the vertex $1$ is an automorphism of $G$.  Because the digraph is connected, $\langle S \rangle=G$, and for an automorphism $\alpha$ that fixes the vertex $1$ to respect the partition $\mathcal B$ means precisely that for any $s_i, s_j \in S$ we have $\alpha(s_is_j)=\alpha(s_i)\alpha(s_j)$.  Similarly for longer words from $\langle S \rangle$.  In the case of graphs, the proof becomes more complicated since respecting the partition means only that $\alpha(s_is_j)$ is one of $\alpha(s_i)\alpha(s_j)$, $\alpha(s_i)\alpha(s_j^{-1})$, $\alpha(s_i^{-1})\alpha(s_j)$, or $\alpha(s_i^{-1})\alpha(s_j^{-1})$.  However, the proof of this for circulant graphs is a special case of our main theorem.

It is our main theorem that in the case of circulant graphs and digraphs (Cayley graphs on $\Z_n$), we can show that only group automorphisms of $\Z_n$ respect the partition $\mathcal C$ while fixing the vertex $0$.

This question was suggested by Toma\v{z} Pisanski.  It arose in the context of studying the structure and automorphism groups of {\it GI}-graphs, a generalisation of both the class of generalised Petersen graphs and the Foster census {\it I}-graphs (see \cite{CPZ}). The question seemed to me to be of interest in its own right.

\section{Main Theorem and Proof}

A \emph{Cayley digraph} $\Cay(G;S)$ for a group $G$ and a subset $S \subset G$ with $1 \not\in S$, is the digraph whose vertices correspond to the elements of $G$, with an arc from $g$ to $gs$ whenever $g \in G$ and $s \in S$.  If $S$ is closed under inversion, then we combine the arcs from $g$ to $gs$ and from $gs$ to $gss^{-1}=g$ into a single undirected edge, and the resulting structure is a \emph{Cayley graph}.  A \emph{circulant (di)graph} $\Circ(n;S)$ is a Cayley (di)graph on the group $G=\Z_n$.

We introduce some notation that will be useful in our proof. For this notation, we assume that $
\Gamma=\Circ(n;S)$ is fixed, with $S=\{ s_1, \ldots, s_c\}$.  For any $k$, we will use $S_k$ to denote $\langle s_1, \ldots, s_k \rangle$.

We begin with some lemmas. Notice that since the circulant graph is defined on a cyclic group, we will be using additive notation for this group.

\begin{lem}\label{multiplier}
Let $\alpha \in \Aut(\Gamma)$ respect $\mathcal C$ and fix the vertex labelled $0$. Suppose $s, s' \in S$ and $\alpha(s) \equiv js\pmod{n_1}$, where $n=n_1n_2$, $\gcd(n_1,n_2)=1$, and $\langle n_2 \rangle \le \langle s \rangle$. Then $\alpha(s')\equiv js' \pmod{n_1}$.
\end{lem}

\begin{proof}
Let $m$ be such that $mn_2 \equiv 1 \pmod{n_1}$; such an $m$ exists since $\gcd(n_1,n_2)=1$, and we will also have $\gcd(m,n_1)=1$. Since $\alpha$ respects $\mathcal C$, we have $\alpha(as)\equiv ajs\pmod{n_1}$, for any $a \in \mathbb Z$. In particular, $\alpha(amn_2s)\equiv amn_2js \equiv ajs\pmod{n_1}$, for any integer $a$.

Since $\langle n_2 \rangle \le \langle s\rangle$, there is some $t$ such that $st=n_2$, so $st$ has order $n_1$ in $\mathbb Z_n$. By the definition of $m$, we see that $mn_2st \equiv st \pmod{n_1}$, so  has order $n_1$ in $\mathbb Z_n$. Thus every element of $\langle n_2 \rangle$ can be written as a multiple of $mn_2s$.

Consider $\alpha(mn_2s')$. Clearly $mn_2s' \in \langle n_2 \rangle$, so $mn_2s'=xmn_2s$ for some integer $x$. Now $\alpha(mn_2s')=\alpha(xmn_2s)\equiv xmn_2js\pmod{n_1}$ by the conclusion of the first paragraph of this proof. Furthermore, this is $mn_2js'$.  Since $\alpha$ respects $\mathcal C$, we know that $mn_2\alpha(s')=\alpha(mn_2s')\equiv mn_2js'\pmod{n_1}$.  Since $m$ and $n_2$ are coprime to $n_1$, this implies that $\alpha(s') \equiv js'\pmod{n_1}$, as desired.
\end{proof}

The next lemma follows from the first. We will be using notation that was introduced by Godsil in \cite{Godsil} and has become standard: $\Aut(G;S)$ denotes the automorphisms of the group $G$ that fix $S$ setwise, where $S \subseteq G$.

\begin{lem}\label{j}
Assume $\Gamma$ is connected. Let $\alpha \in \Aut(\Gamma)$ respect $\mathcal C$ and fix the vertex labelled $0$. Then there is some $\beta \in \Aut(\Z_n;S)$ such that $\beta\alpha$ fixes the vertex $as$ for every $a \in \Z$ and every $s \in S$.
\end{lem}

\begin{proof}
Let $n=p_1^{e_1}\ldots p_r^{e_r}$, where $p_1, \ldots, p_r$ are distinct primes.  For any $p_i$, since $\Gamma$ is connected, there is some $s_i \in S$ such that $p_i \nmid s_i$, and hence $\langle n/p_i^{e_i}\rangle \le \langle s_i \rangle$.  Let $j_i$ be such that $\alpha(s_i) \equiv j_is_i \pmod{p_i^{e_i}}$. Notice that $j_i \not\equiv 0 \pmod{p_i}$ since $\alpha$ respecting $\mathcal C$ implies that both $s_i$ and $\alpha(s_i)$ have the same order in $\Z_n$. Thus the conditions of Lemma \ref{multiplier} are satisfied, and we conclude that for any $s \in S$ we have $\alpha(s) \equiv j_is \pmod{p_i^{e_i}}$.

Let $j$ be such that $j \equiv j_i\pmod{p_i^{e_i}}$ for every $1 \le i \le r$, and $1 \le j \le n-1$; by the Chinese Remainder Theorem, such a $j$ exists. Then for every $s \in S$, we have $\alpha(s) = js$, since this is the only value that satisfies all of the congruences. Furthermore, since $\alpha$ respects $\mathcal C$, this means that for every $a \in \Z$ and every $s \in S$, $\alpha(as)=ajs$. 

Let $\beta \in \Aut(\Z_n)$ be the automorphism that corresponds to multiplication by $j^{-1}$. (Since $j \not\equiv 0\pmod{p_i}$ for any $i$, $j \in \Z_n^*$ has an inverse, and this is an automorphism of $\Z_n$.) Since $\alpha$ is an automorphism of $\Gamma$ and $\alpha(s)=js$ for every $s \in S$, we see that multiplication by $j$ fixes $S$ setwise, so $\beta \in \Aut(Z_n;S)$. Clearly $\beta\alpha$ fixes $as$ for every $a \in \Z$ and every $s \in S$.
\end{proof}

The next lemma is an easy consequence of the definition of respecting $\mathcal C$, but is very useful.

\begin{lem}\label{cosets}
Let $G=\langle S' \rangle$ for some $S' \subseteq S$. Let $\alpha \in \Aut(\Gamma)$ respect $\mathcal C$. For any $x \in \Z_n$, $\alpha(x+G)$ is a coset of $G$.
\end{lem}

\begin{proof}
For any $s \in S$ we have $x$ and $x+s$ are together in a cycle $C$ of length $|s|$. Since $\alpha$ respects $\mathcal C$, $\alpha(C)$ is also a cycle of length $|s|$. Since $\Z_n$ has a unique subgroup of order $|s|$, $\alpha(C)$ must be a coset of this subgroup.  Suppose $\alpha(x)=y$, then $\alpha(C)=y+\langle s\rangle$, and this is true for every $s \in S'$, so $\alpha(x+G)=y+G$.
\end{proof}

We can use Lemma~\ref{j} to assume that many of the vertices of $\Gamma$ are fixed by $\alpha$, specifically vertices of the form $as$ where $a \in \mathbb Z$ and $s \in S$.  In our next lemma, we show that if some vertices are known to be fixed by a graph automorphism that respects $\mathcal C$, this will force other vertices to be fixed also. This lemma is technical, but is the very core of the proof of our main theorem.

\begin{lem}\label{4-cycles}
Let $G=\langle S' \rangle$ for some $S' \subseteq S$. Let $|G|=n'$, let $s \in S$ with $|s|=r$, and let $d=\gcd(n',r)$. Suppose that $\alpha \in\Aut(\Gamma)$ fixes every vertex of some set $T$, where $G \subseteq T \subseteq G'=\langle G, s\rangle$, and $T$ is a union of cosets of $\langle n/d \rangle$. If $x, x+s', x+s \in T$ with $s' \in S'$, then $\alpha$ fixes $x+s+s'$.
\end{lem}

\begin{proof}
By assumption, $\alpha$ fixes $x$, $x+s'$, and $x+s$. In $G'$, every coset of $G$ contains at least one vertex of $\langle s\rangle$. Since this vertex is fixed by $\alpha$, by Lemma~\ref{cosets} we have that every coset of $G$ in $G'$ is fixed setwise by $\alpha$.  Similarly, every coset of $\langle s\rangle$ in $G'$ contains at least one vertex of $G$, and hence is fixed setwise by $\alpha$.  Hence every intersection of a coset of $G$ with a coset of $\langle s \rangle$ is fixed (setwise) by $\alpha$; that is, every coset of $\langle n/d \rangle$ in $G'$ is fixed setwise by $\alpha$. If $d=1$ then these cosets are all singletons, one of which is $x+s+s'$, and we are done. We therefore assume $d>1$.

Since the coset of $\langle n/d \rangle$ that contains $x+s+s'$ is fixed setwise by $\alpha$, we must have $\alpha(x+s+s')=x+s+s'+z(n/d)$ for some $z<d$. If $z=0$ then we are done, so we suppose $0 <z<d$.  

Choose $p$ prime and $a \in \Z$ such that $p^a \mid d$ but $p^a \nmid z$; such a $p$ and $a$ exist because $0<z<d$.

Since $\alpha$ respects $\mathcal C$, fixes $x+s'$, and takes $x+s'+s$ to $x+s'+s+z(n/d)$, we must have $\alpha(x+s'+bs)=x+s'+b(s+z(n/d))$ for any integer $b$. In particular, when $b=r/d$, we get $\alpha(x+s'+r/d(s))=x+s'+r/d(s+z(n/d))$. Now, since $|s|=r$ in $\Z_n$, we must have $s=\ell(n/r)$ for some $\ell$ coprime to $n$. Thus, $(r/d)s=(r/d)\ell(n/r)=\ell(n/d)$. Since $x+s' \in T$ and $T$ is a union of cosets of $\langle n/d\rangle$, this shows that $x+s'+r/d(s) \in T$, so by assumption $\alpha$ fixes $x+s'+r/d(s)$. Hence $(r/d)z(n/d)\equiv 0 \pmod{n}$, so we must have $d \mid z(r/d)$.  In particular, $p^a$ divides $z(r/d)$, and since $p^a \nmid z$, this means $p \mid r/d$.

Similarly, since $\alpha$ respects $\mathcal C$, fixes $x+s$, and takes $x+s+s'$ to $x+s+s'+z(n/d)$, we must have $\alpha(x+s+bs')=x+s+b(s'+z(n/d))$ for any integer $b$. In particular, when $b=n'/d$, we get $\alpha(x+s+n'/d(s'))=x+s+n'/d(x'+z(n/d))$. Since $|G|=n'$ is cyclic and $s' \in G$, we have $s'=k(n/n')$ for some $k$. Thus, $(n'/d)s'=(n'/d)k(n/n')=k(n/d)$. Since $x+s \in T$ and $T$ is a union of cosets of $\langle n/d\rangle$, this shows that $x+s+n'/d(s') \in T$, so by assumption $\alpha$ fixes $x+s+n'/d(s')$. Hence $(n'/d)z(n/d)\equiv 0 \pmod{n}$, so we must have $d \mid z(n'/d)$.  In particular, $p^a$ divides $z(n'/d)$, and since $p^a \nmid z$, this means $p \mid n'/d$.

This contradicts the definition of $d=\gcd(n',r)$, so we must have $z=0$, and hence $\alpha(x+s+s')=x+x+s'$.
\end{proof}

We are now ready to prove our main theorem.  

\begin{thm}\label{main}
Let $\Gamma=\Circ(n;S)$ be a connected circulant graph.
Let $\alpha\in\Aut(\Gamma)$ fix the vertex $0$ and respect the partition $\mathcal C$, so for any $C \in \mathcal C$, $\alpha(C) \in \mathcal C$.  Then $\alpha \in \Aut(\Z_n)$.
\end{thm}

\begin{proof}
By Lemma~\ref{j}, replacing $\alpha$ by $\beta\alpha$ if necessary, we may assume that $\alpha$ fixes $as$ for every $a \in \Z$ and every $s\in S$. We will show that $\alpha$ in fact fixes every vertex of $\Gamma$, so $\alpha=1 \in \Aut(\Z_n)$.

We will proceed with a nested induction argument in order to prove that every vertex of $\Gamma$ is fixed by $\alpha$.  In the outer induction we will prove that for each $i$, every vertex of $S_i$ is fixed by $\alpha$. For our base case, we know that every vertex of $S_1=\langle s_1\rangle$ is fixed by $\alpha$, as every vertex of $\langle s \rangle$ is fixed by $\alpha$ for every $s \in S$.  Inductively, assume that every vertex of $S_k$ is fixed by $\alpha$. We will deduce that every vertex of $S_{k+1}$ is fixed by $\alpha$.

Define $T_0=S_k \cup \langle s_{k+1}\rangle$, and for $m \ge 1$, $$T_m=T_{m-1}\cup \{s \in S_{k+1}: s-s_{k+1} \in T_{m-1}\text{ and }s-s_y \in T_{m-1}\text{ for some }1\le y \le k\}.$$  It is not hard to see that every element of $S_{k+1}$ will be in $T_m$ for some $m$.  Our inner inductive argument will be to show that for each $i$, every vertex in $T_i$ is fixed. Clearly, since $S_k$ is fixed pointwise by $\alpha$ by our outer inductive hypothesis, and every vertex of $\langle s_{k+1}\rangle$ is fixed by $\alpha$, every vertex of $T_0$ is also fixed by $\alpha$. This is the base case for our inner induction.

Notice that $T_0$ is a union of cosets of $\langle n/d \rangle$.  We claim that every $T_m$ is a union of cosets of $\langle n/d\rangle$.  We prove this by yet another inductive argument, before we begin our proof that every vertex of $T_m$ is fixed by $\alpha$, as it will be required in that proof.  Suppose that $x' \in T_m$.  If $x' \in T_{m-1}$ then by our inductive hypothesis, the coset of $\langle n/d \rangle$ that contains $x'$ is in $T_{m-1}$.  If $x' \not\in T_{m-1}$ then $x'-s_{k+1} \in T_{m-1}$ and there is some $1 \le y \le k$ such that $x'-s_y \in T_{m-1}$.  But since $T_{m-1}$ is a union of cosets of $\langle n/d\rangle$, this means that $x'-s_{k+1}+\langle n/d\rangle \subseteq T_{m-1}$ and $x'-s_y+\langle n/d\rangle \subseteq T_{m-1}$, so clearly $x'+\langle n/d \rangle \subseteq T_m$, as desired.

Now we proceed with our main inner inductive argument, to show that $\alpha$ fixes every point of $S_{k+1}$.
Suppose that every vertex in $T_m$ is fixed by $\alpha$.  Let $x'$ be an arbitrary vertex of $T_{m+1}$.  If $x'\in T_m$ then $\alpha$ fixes $x'$ by hypothesis and we are done.  So by the definition of $T_{m+1}$, we have $x'-s_y \in T_m$ for some $1\le y \le k$, and inductively either $x'-s_y-s_{k+1} \in T_{m_1}$ for some $m_1 \le m-1$, or $x'-s_y \in T_0$.  If $x'-s_y \in \langle s_{k+1}\rangle$ then $x'-s_y-s_{k+1}\in T_0 \subset T_m$, while if $x'-s_y \in S_k$ then $x' \in S_k$ is fixed by $\alpha$ and we are done.  So we may assume that $x=x'-s_y-s_{k+1} \in T_m$, as well as $x+s_{k+1}=x'-s_y \in T_m$ and $x+s_y=x'-s_{k+1} \in T_m$. 

We appeal to Lemma~\ref{4-cycles}, with $G=S_k$, $s=s_{k+1}$, and $T=T_m$. Since all of the conditions of the lemma are satisfied, we conclude that $\alpha$ fixes $x+s_y+s_{k+1}=x'$. Thus every vertex of $T_{m+1}$ is fixed by $\alpha$. This completes the inner induction, allowing us to conclude that every vertex of $S_{k+1}$ is fixed by $\alpha$, which completes the outer induction and the proof.
\end{proof}

\section{Acknowledgements}

I am very much indebted to Toma\v{z} Pisanski for suggesting this question.


\end{document}